\documentclass[12pt]{article}
\usepackage{geometry}
\usepackage{graphicx}
\usepackage{amssymb}
\usepackage{amsmath}
\usepackage[colorlinks,linkcolor=blue,citecolor=blue,urlcolor=blue,
            bookmarks,bookmarksnumbered]{hyperref}
\usepackage{amsthm}
\usepackage{tikz-cd}
\theoremstyle{definition}
\newtheorem{dfn}[subsection]{Definition}
\theoremstyle{definition}
\newtheorem{thm}[subsection]{Theorem}
\theoremstyle{definition}
\newtheorem*{conv}{Conventions}

\renewcommand{\phi}{\varphi}

\newcommand{\id}{\text{id}}

\newcommand{\Hom}{\text{Hom}}

\newcommand{\isom}{\approx}
\newcommand{\C}{\mathcal{C}}

\newcommand{\AAlg}{A\text{-Alg}}

\newcommand{\Ab}{\textup{Ab}}
\newcommand{\ep}{\varepsilon}
\newcommand{\Z}{\mathbb{Z}}
\newcommand{\Q}{\mathbb{Q}}
\usepackage{xurl}

\title{Beck torsors, formally unramified objects, and K{\"a}hler differentials}
\author{Nicholas Mertes\footnote{n.mertes@umiami.edu}\\ \textit{Department of Mathematics,}\\\textit{University of Miami},\\\textit{Coral Gables, FL}}
\date{April 2021}							% Activate to display a given date or no date

\begin{document}
\maketitle

\begin{abstract}
Let $\C$ be a category with pullbacks. We define a \textit{Beck torsor} in $\C$ as a morphism $Z\to Y$ in $\C$ which is a torsor for a Beck module over $Y$. We say that an object $X$ of $\C$ is \textit{formally unramified} if, for every Beck torsor $Z\to Y$ in $\C$, the canonical map $\Hom_\C(X, Z)\to \Hom_\C(X, Y)$ is injective. If $A$ is a commutative ring with identity, then an $A$-algebra $B$ is formally unramified in the category of $A$-algebras if and only if the ring homomorphism $A\to B$ is formally unramified. Given that $A\to B$ is formally unramified if and only if $\Omega_{B/A} = 0$, we seek a similar classification for general formally unramified objects. We say that $\C$ has \textit{K{\"a}hler differentials} if, for each object $X$ of $\C$, the forgetful functor $\Ab(\C/X)\to \C/X$ from the category of Beck modules over $X$ has a left adjoint $\Omega: \C/X\to \Ab(\C/X)$. Our main result is that if $\C$ has K{\"a}hler differentials, then an object $X$ of $\C$ is formally unramified if and only if $\Omega_X$ is a zero object in $\Ab(\C/X)$.
\end{abstract}

\section{Introduction}

The study of integer solutions to polynomial equations has led to the development of rich mathematical theories, such as those used in effort to study Fermat equations of the type $x^n + y^n = z^n$. However, for an equation such as $x^2= 0$ whose only integer solution is $x=0$, one is compelled to explore non-integer solutions. For example, in the ring $\Z[\ep]/(\ep^2)$, we obtain the non-zero solution $\ep^2 = 0$. This kind of solution, known as an infinitesimal solution, also plays a significant role in mathematics. For example, the fundamental notions of formally unramified, formally etale, and formally smooth morphisms in algebraic geometry are expressed in terms of lifting criteria inspired by these infinitesimal solutions.

The goal of this paper is to initiate a program to extend the notions of formally unramified, formally etale, and formally smooth morphisms into a categorical setting in effort to begin developing algebraic geometry uniformly over any algebraic theory. The notion of an algebraic theory, which we take to be a category of algebras of a Lawvere theory \cite{Lawvere, nLabLawvere}, includes more than traditional algebraic objects. For example, there is an algebraic theory of $C^\infty$-rings \cite{MSIA} which contains all of classical differential geometry. Therefore, one potential application of our program would be to exhibit immersions, local diffeomorphisms, and submersions in differential geometry as precisely (not just analogously) the same phenomena as unramified, etale, and smooth morphisms in algebraic geometry. It appears that Schreiber \cite{Schreiber1, Schreiber2} makes a similar claim in the context of cohesive $\infty$-topoi, in the sense that he relates an abstract version of these morphisms to the traditional differential-geometric morphisms. However, the categorical interpretation of these morphisms appears to be so fundamental that such elaborate formalism is not needed to develop the core concepts. More speculatively, our program could be used to explore geometric objects over algebraic theories in which geometry is not typically formulated. For example, what is an algebraic stack over the category of groups?

In this paper, we focus only on formally unramified morphisms and defer a discussion of formally etale and formally smooth morphisms to a future work. In section \ref{Beck_torsors}, we review the notion of a Beck module and introduce Beck torsors which are a categorical interpretation of the classical notion of square-zero extensions. With this categorical version of square-zero extensions at hand, we introduce formally unramified objects in section \ref{Formally_unramified_objects}. It appears that Paugam \cite{Paugam} considered a similar definition of formally unramified objects in the context of sheaves on a site, but that definition is not referred to elsewhere in his book.

It is important to make sure that our definition of formally unramified objects is justified to carry the name. What properties do formally unramified objects share with formally unramified ring homomorphisms? A classical result in commutative algebra states that a ring homomorphism $A\to B$ is formally unramified if and only if $\Omega_{B/A} = 0$, where $\Omega_{B/A}$ is the module of K{\"a}hler differentials. Provided a category $\C$ is sufficiently nice (such as an algebraic theory), it is possible to define a Beck module $\Omega_X$ of K{\"a}hler differentials for each object $X$ of $\C$. In section \ref{Kahler_differentials}, we discuss this general notion of K{\"a}hler differentials and prove that $X$ is a formally unramified object of $\C$ if and only if $\Omega_X$ is a zero object in the category of Beck modules over $X$. This theorem provides some justification that our definition of formally unramified objects makes sense.

\begin{conv}
A \textit{ring} is a commutative ring with identity. A \textit{ring homomorphism} is an identity-preserving ring homomorphism. We fix a ring $A$ throughout the paper. An \textit{$A$-algebra} is a ring $B$ together with a ring homomorphism $\mu_B: A\to B$. If $B$ and $C$ are $A$-algebras, then a ring homomorphism $\phi: B\to C$ is called an \textit{$A$-algebra homomorphism} if $\phi\circ\mu_B = \mu_C$. We write $\AAlg$ for the category of $A$-algebras and $A$-algebra homomorphisms.
\end{conv}

\section{Beck torsors}\label{Beck_torsors}

We want to reformulate the notion of a formally unramified ring homomorphism in a way which can be interpreted in any category with pullbacks. We begin by recalling the traditional definition of square-zero extensions and formally unramified ring homomorphisms.

\begin{dfn}
An $A$-algebra homomorphism $\phi: D\to C$ is called a \textit{square-zero extension} if $\phi$ is surjective and $(\ker\phi)^2 = 0$.
\end{dfn}
\begin{dfn}
Let $B$ be an $A$-algebra. We say that $\mu_B: A\to B$ is \textit{formally unramified} (or simply that $B$ is formally unramified) if, for every square-zero extension $D\to C$ of $A$-algebras, the canonical map
\[
\Hom_{\AAlg}(B, D)\to \Hom_{\AAlg}(B, C)
\]
is injective.
\end{dfn}

The property of $\Hom_{\AAlg}(B, D)\to \Hom_{\AAlg}(B, C)$ being injective already makes sense for any category, but generalizing the notion of square-zero extension takes some work. The fundamental observation that makes this generalization possible is that $D\to C$ is a square-zero extension if and only if $D\to C$ is a torsor for a Beck module over $C$. Then we note that this kind of torsor can be interpreted in any category with pullbacks. We now make these claims precise.

\begin{dfn}
Let $\C$ be a category with pullbacks. Let $X$ be an object of $\C$. A \textit{Beck module} over $X$ is a morphism $M\to X$ in $\C$ which is an abelian group object in the overcategory $\C/X$. A morphism of Beck modules over $X$ is a morphism of abelian group objects. We write $\Ab(\C/X)$ for the category of Beck modules over $X$.
\end{dfn}

We require $\C$ to have pullbacks because this assumption implies that $\C/X$ has products. Since $\C/X$ automatically has a terminal object $\id_X: X\to X$, we therefore have sufficient structure to interpret group objects in $\C/X$.

A result due to Beck \cite{Beck} states that if $B$ is an $A$-algebra, then $\Ab(\AAlg/B)$ is equivalent to the category of $B$-modules. The main point of this equivalence is that every Beck module over $B$ is a split square-zero extension (isomorphic to a square-zero extension of the type $B\oplus M\to B$ for a $B$-module $M$). But what can we say about general square-zero extensions? In a previous work \cite{Cotorsors}, we showed that general square-zero extensions are the torsors for split square-zero extensions.\footnote{We actually showed that this is true for the category of rings rather than for the category of $A$-algebras. However, the details of the proof for $A$-algebras are the same provided that ``ring homomorphism" is replaced throughout with ``$A$-algebra homomorphism."} We now define torsors for Beck modules in a general category with pullbacks. This is possible because Beck modules are group objects.
\begin{dfn}
Let $\C$ be a category with pullbacks. A \textit{Beck torsor} in $\C$ is an effective epimorphism $Z\to Y$ in $\C$ together with a Beck module $M$ over $Y$ and a group action $\tau: M\times_Y Z\to Z$ in $\C/Y$ such that $(\tau, \pi_Z): M\times_Y Z\to Z\times_Y Z$ is an isomorphism in $\C/Y$.
\end{dfn}
Note that a Beck module $M\to Y$ is a torsor for itself, with group action given by the group addition $+: M\times_Y M\to M$. The requirement that $Z\to Y$ is an effective epimorphism means that
\[
Z\times_Y Z\rightrightarrows Z\to Y
\]
is a coequalizer diagram. The reason for this condition is that an effective epimorphism (not merely an epimorphism) is the correct analogue of a surjection of rings. Indeed, a ring homomorphism such as $\Z\to\Q$ is an epimorphism but not surjective. However, a ring homomorphism is surjective if and only if it is an effective epimorphism. Since square-zero extensions are surjective ring homomorphisms, using the condition of effective epimorphism in the definition of Beck torsors seems natural.

For later use, we record what the torsor property tell us on the level of Hom sets.
\begin{thm}\label{bijection_for_the_torsor_action}
Let $\C$ be a category with pullbacks. Let $Z\to Y$ be a Beck torsor in $\C$ with Beck module $M$ and group action $M\times_Y Z\to Z$. Then we have a bijection
\[
\Hom_\C(X, M)\times_{\Hom_\C(X, Y)} \Hom_\C(X, Z)\isom \Hom_\C(X, Z)\times_{\Hom_\C(X, Y)} \Hom_\C(X, Z)
\]
\end{thm}
\begin{proof}
Since $Z\to Y$ is a Beck torsor, we have that $M\times_Y Z\isom Z\times_Y Z$. The theorem then follows from the universal property of pullbacks. Indeed,
\[
\begin{split}
\Hom_\C(X, M)\times_{\Hom_\C(X, Y)} \Hom_\C(X, Z) &\isom \Hom_\C(X, M\times_Y Z) \\
&\isom \Hom_\C(X, Z\times_Y Z) \\
&\isom \Hom_\C(X, Z)\times_{\Hom_\C(X, Y)} \Hom_\C(X, Z).
\end{split}
\]
\end{proof}

\section{Formally unramified objects}\label{Formally_unramified_objects}

We now use Beck torsors as a categorical generalization of square-zero extensions to define formally unramified objects.

\begin{dfn}
Let $\C$ be a category with pullbacks. Let $X$ be an object of $\C$. We say that $X$ is \textit{formally unramified} if, for every Beck torsor $Z\to Y$ in $\C$, the canonical map $\Hom_\C(X, Z)\to \Hom_\C(X, Y)$ is injective.
\end{dfn}
Given that Beck torsors in $\AAlg$ coincide with square-zero extensions, it is clear that an $A$-algebra $B$ is a formally unramified object of $\AAlg$ if and only if the ring homomorphism $\mu_B: A\to B$ is formally unramified.

We now discuss an aspect of formally unramified ring homomorphisms which is sometimes not emphasized in the literature, but is important for this paper. Suppose that $B$ is a formally unramified $A$-algebra. Then, for every Beck module $C\oplus M\to C$ in $\AAlg$, the canonical map
\[
\Hom_{\AAlg}(B, C\oplus M)\to \Hom_{\AAlg}(B, C)
\]
is \textit{bijective}. Said differently, the lifting exists and is unique with respect to the \textit{split} square-zero extensions. So what distinguishes formally unramified ring homomorphisms from formally etale ring homomorphisms\footnote{Recall that a ring homomorphism is formally etale if the lifting exists and is unique with respect to \textit{all} square-zero extensions.} is the existence of square-zero extensions which are not split, that is, the existence of Beck torsors which are not Beck modules.

We now prove that this bijectivity with respect to Beck modules holds in general for formally unramified objects. Furthermore, it is an if and only if statement, meaning that this bijectivity fully characterizes formally unramified objects.
\begin{thm}\label{bijection_for_beck_modules}
Let $\C$ be a category with pullbacks. Let $X$ be an object of $\C$. Then $X$ is formally unramified if and only if, for each Beck module $M\to Y$ in $\C$, the canonical map
\[
\Hom_\C(X, M)\to \Hom_\C(X, Y)
\]
is bijective.
\end{thm}
\begin{proof}
First assume that $X$ is formally unramified. Let $\phi: M\to Y$ be a Beck module in $\C$. Since every Beck module is a Beck torsor, we know that
\[
\Hom_\C(X, M)\to \Hom_\C(X, Y)
\]
is injective. It remains to show that this map is surjective. Let $\alpha: X\to Y$ be a morphism in $\C$. Since $\phi$ is a Beck module, $\phi$ is an abelian group object in $\C/Y$. Note that $\id_Y: Y\to Y$ is a terminal object of $\C/Y$, and hence part of the abelian group structure on $\phi$ is a unit $u_M: Y\to M$ such that $\phi\circ u_M = \id_Y$. Define $\psi: X\to M$ such that $\psi = u_M\circ\alpha$. Then the image of $\psi$ under $\Hom_\C(X, M)\to \Hom_\C(X, Y)$ is $\phi\circ \psi = \phi\circ u_M\circ\alpha = \id_Y\circ\alpha = \alpha$. Therefore, the map is surjective.

Now assume that, for each Beck module $M\to Y$ in $\C$, the canonical map
\[
\Hom_\C(X, M)\to \Hom_\C(X, Y)
\]
is bijective. Let $\gamma: Z\to Y$ be a Beck torsor with Beck module $M$ and group action $M\times_Y Z\to Z$. To show that $X$ is formally unramified, we need to show that the canonical map
\[
\Hom_\C(X, Z)\to \Hom_\C(X, Y)
\]
is injective. Let $\alpha, \beta\in \Hom_\C(X, Z)$ be such that $\gamma\circ\alpha = \gamma\circ\beta$. Then
\[
(\alpha, \beta)\in \Hom_\C(X, Z)\times_{\Hom(X, Y)}\Hom_\C(X, Z)
\]
Using Theorem \ref{bijection_for_the_torsor_action} and the assumed bijection $\Hom_\C(X, M)\isom \Hom_\C(X, Y)$, we have that
\[
\begin{split}
\Hom_\C(X, Z)\times_{\Hom(X, Y)}\Hom_\C(X, Z)&\isom \Hom_\C(X, M)\times_{\Hom_\C(X, Y)} \Hom_\C(X, Z) \\
&\isom \Hom_\C(X, Y)\times_{\Hom_\C(X, Y)} \Hom_\C(X, Z) \\
&\isom \Hom_\C(X, Y\times_Y Z) \\
&\isom \Hom_\C(X, Z),
\end{split}
\]
where the final bijection $\Hom_\C(X, Z)\times_{\Hom(X, Y)}\Hom_\C(X, Z)\to \Hom_\C(X, Z)$ maps $(\alpha, \beta)$ to $\beta$. Now, since $(\beta, \beta)$ also maps to $\beta$, we have that $(\alpha, \beta) = (\beta, \beta)$ and hence $\alpha = \beta$.
\end{proof}

\section{K{\"a}hler differentials}\label{Kahler_differentials}

Finally, we discuss K{\"a}hler differentials and their relationship with formally unramified objects. This categorical definition of K{\"a}hler differentials was introduced by Quillen \cite{Quillen_book, Quillen_notes}. We first reformulate the ordinary notion of K{\"a}hler differentials in a way which suggests the categorical generalization.

Let $B$ be an $A$-algebra. We have a forgetful functor $\Ab(\AAlg/B)\to \AAlg/B$. Under the identification of $\Ab(\AAlg/B)$ with the category of $B$-modules, the left adjoint $\Omega: \AAlg/B\to \Ab(\AAlg/B)$ is such that, for each object $C\to B$ of $\AAlg/B$, $\Omega(C\to B) = \Omega_{C/A}\otimes_C B$. In particular, $\Omega(\id_B) = \Omega_{B/A}\otimes_B B\isom \Omega_{B/A}$. This classical result motivates the following definition in general.

\begin{dfn}
Let $\C$ be a category with pullbacks. We say that $\C$ has \textit{K{\"a}hler differentials} if, for each object $X$ of $\C$, the forgetful functor $\Ab(\C/X)\to \C/X$ has a left adjoint $\Omega: \C/X\to \Ab(\C/X)$. We write $\Omega_X = \Omega(\id_X)$ and refer to $\Omega_X$ as the \textit{module of K{\"a}hler differentials} of $X$.
\end{dfn}

By the adjunction of $\Omega$ and the forgetful functor, we have a bijection
\[
\Hom_{\Ab(\C/X)}(\Omega_X, M)\isom \Hom_{\C/X}(X, M)
\]
for each object $X$ of $\C$ and each Beck module $M$ over $X$. We are now ready to state and prove the main result of this paper.

\begin{thm}
Let $\C$ be a category with K{\"a}hler differentials. Let $X$ be an object of $\C$. Then $X$ is formally unramified if and only if $\Omega_X$ is a zero object in $\Ab(\C/X)$.
\end{thm}
\begin{proof}
First assume that $\Omega_X$ is a zero object in $\Ab(\C/X)$. To show that $X$ is formally unramified, Theorem \ref{bijection_for_beck_modules} tells us it suffices to show that, for each Beck module $M\to Y$ in $\C$, the canonical map
\[
\Hom_\C(X, M)\to \Hom_\C(X, Y)
\]
is bijective. Let $\phi: M\to Y$ be a Beck module in $\C$. Let $\psi\in\Hom_\C(X, Y)$. We want to show that there exists a unique $\alpha\in \Hom_\C(X, M)$ such that $\phi\circ\alpha = \psi$. Consider the pullback Beck module $\psi^\ast M = X\times_Y M$ over $X$. Then
\[
\Hom_{\Ab(\C/X)}(\Omega_X, \psi^\ast M)\isom \Hom_{\C/X}(X, \psi^\ast M).
\]
Since $\Omega_X$ is a zero object in $\Ab(\C/X)$, $\Hom_{\Ab(\C/X)}(\Omega_X, \psi^\ast M)$ has exactly one element and hence $\Hom_{\C/X}(X, \psi^\ast M)$ has exactly one element. The universal property of the pullback digram
\[
\begin{tikzcd}
  X
  \arrow[drr, bend left, "\alpha"]
  \arrow[swap, ddr, bend right, "\id_X"]
  \arrow[dr, dotted] & & \\
    & \psi^\ast M \arrow[r] \arrow[d]
      & M \arrow[d, "\phi"] \\
& X \arrow[r, "\psi"] & Y
\end{tikzcd}
\]
tells us that there exists a unique $\alpha: X\to M$ such that $\phi\circ\alpha = \psi\circ\id_X = \psi$.

Now assume that $X$ is formally unramified. We want to show that $\Omega_X$ is a zero object of $\Ab(\C/X)$. Since $\id_X: X\to X$ is a zero object, it suffices to show that $\Omega_X\isom \id_X$. Since any two initial objects are isomorphic, it suffices to show that $\Omega_X$ is initial. Let $M\to X$ be a Beck module over $X$. Since $X$ is formally unramified, Theorem \ref{bijection_for_beck_modules} tells us that
\[
\begin{split}
\Hom_{\Ab(\C/X)}(\Omega_X, M) &\isom \Hom_{\C/X}(X, M) \\
&\isom \Hom_{\C/X}(X, X).
\end{split}
\]
Since $\Hom_{\C/X}(X, X)$ has exactly one element, $\Hom_{\Ab(\C/X)}(\Omega_X, M)$ has exactly one element and hence $\Omega_X$ is an initial object of $\Ab(\C/X)$.
\end{proof}

\section{Conclusion}

We showed that if $\C$ is a category with K{\"a}hler differentials, then an object $X$ of $\C$ is formally unramified if and only if $\Omega_X$ is a zero object in the category of Beck modules over $X$. We did not address the issue of which categories actually have K{\"a}hler differentials. The good news is that any locally presentable category, such as an algebraic theory (a category of algebras of a Lawvere theory), has K{\"a}hler differentials. This can be shown by using the locally presentable adjoint functor theorem \cite{LPAC, nLabAdjointFunctor}. So, if our goal for now is to develop algebraic geometry over algebraic theories, then we need not be concerned with the issue of existence of K{\"a}hler differentials. However, it would be interesting in the future to find an exhaustive classification of which categories have K{\"a}hler differentials.

Moving forward, an attempt will soon be made to introduce and study formally etale and formally smooth objects in a manner similar to what has been done in the current paper. The notion of finite presentation also makes sense categorically, meaning that we can define traditional unramified, etale, and smooth objects (without the ``formally"). Long term projects include finding out whether immersions, local diffeomorphisms, and submersions can be described in this language via the algebraic theory of $C^\infty$-rings, and also using these categorical definitions to introduce and study analogues of algebraic spaces and stacks over algebraic theories.

\section{Acknowledgements}

I originally had the idea for this paper while reading the mathoverflow question \cite{Overflow}. Indeed, as a special case of the main theorem of this paper, we answer this question by achieving a purely categorical proof of the classical fact that a ring homomorphism $A\to B$ is formally unramified if and only if $\Omega_{B/A} = 0$.

\end{document}